\documentclass[]{interact}

\usepackage{epstopdf}
\usepackage[caption=false]{subfig}

\usepackage[numbers,sort&compress]{natbib}
\bibpunct[, ]{[}{]}{,}{n}{,}{,}

\makeatletter
\def\NAT@def@citea{\def\@citea{\NAT@separator}}
\makeatother

\usepackage{a4}
\usepackage{amsfonts}  
\usepackage{amsmath}                   
\usepackage{amsthm}
\usepackage{color}
\usepackage{amssymb}
\usepackage{enumerate}
\usepackage{soul}
\usepackage{pgfplots}
\pgfplotsset{compat=1.15}
\usepackage{mathrsfs}
\usetikzlibrary{arrows}

\DeclareMathOperator{\eco}{econv}	

\DeclareMathOperator{\conv}{conv}

\DeclareMathOperator{\cl}{cl}

\DeclareMathOperator{\epco}{\normalfont{e}^\prime conv} 	
 	
\DeclareMathOperator{\ep}{\text{e}^\prime} 	
\DeclareMathOperator{\dom}{dom}					
\DeclareMathOperator{\epi}{epi}

\newcommand{\Ramp}{\overline{\mathbb{R}}}
\newcommand{\R}{\mathbb{R}}

\newcommand{\ci}{\left\langle}		
\newcommand{\cd}{\right\rangle}

\newcommand{\xb}{\overline{x}}			
\newcommand{\yb}{\overline{y}}

\newcommand{\Db}{\overline{D}}

\newcommand{\alphab}{\overline{\alpha}}

\theoremstyle{plain}
\newtheorem{theorem}{Theorem}[section]
\newtheorem{lemma}[theorem]{Lemma}
\newtheorem{corollary}[theorem]{Corollary}
\newtheorem{proposition}[theorem]{Proposition}

\theoremstyle{definition}
\newtheorem{definition}[theorem]{Definition}
\newtheorem{example}[theorem]{Example}

\theoremstyle{remark}
\newtheorem{remark}{Remark}

\begin{document}

\articletype{ARTICLE TEMPLATE}

\title{Fenchel $c$-conjugate dual problems for DC optimization: 
characterizing weak and strong duality}

\author{
\name{M.D. Fajardo\textsuperscript{a}\thanks{CONTACT M.~D. Fajardo. Email: md.fajardo@ua.es} and J. Vidal\textsuperscript{b}}
\affil{\textsuperscript{a}Department of Mathematics, University of Alicante, Alicante, Spain; \textsuperscript{b}Department of Physics and Mathematics, University of Alcalá, Alcalá de Henares, Spain}
}

\maketitle

\begin{abstract}
This paper focuses on the development of two Fenchel-type dual problems, $(D^F)$ and $\left(\Db^F \right )$, for a given DC (difference of convex functions) optimization problem $(P)$ by means of a conjugation scheme suitable for evenly convex functions. We derive conditions characterizing not only weak duality and zero duality gap, but also strong and stable strong dualities for $(P)-(D^F)$ and $(P)-\left(\Db^F \right )$ and their relationships.
\end{abstract}

\begin{amscode}
52A20, 26B25, 90C25, 49N15.
\end{amscode}

\begin{keywords}
Generalized convex conjugation, Evenly convex function, DC problem, Fenchel duality
\end{keywords}

\section{Introduction}
\label{sec:Intro}

Conjugate duality and convex optimization have a wide range of applications solving real optimization problems. In this paper we will concentrate on \emph{DC problems}, i.e., problems whose objective function is the difference of two proper convex functions. For different applications of these problems, see \cite{HT1999,CGR2018,AT2005,TA1997} among others. For a general overview of conjugate duality, we refer the reader to \cite{B2010,Z2002} for theory on general spaces, \cite{BC2011} for a context on Hilbert spaces and \cite{R1970} for the Euclidean case.

The combination of these two areas, DC programming and conjugate duality theory, have attained considerable attention recently, not only from the theoretical point of view, but also from the practical one. For instance and among the huge literature in this topic, we mention \cite{FAY2021}, where the authors develop new regularity conditions for Fenchel duality in DC optimization problems involving composite functions; \cite{LXQ2021}, which focuses on (weak/zero/strong) dualities in DC infinite programming with inequality constraints; in the same framework, \cite{S2014} provides regularity conditions and some Farkas-type results. Quite recently, \cite{FV2022} develops some techniques based on subdifferentials via the $c$-conjugation scheme, defined in \cite{MLVP2011}, with an application into optimality conditions for DC problems, extending in this way some results on subdifferentials based on Fenchel conjugation from \cite{HU1988}.

In every approach including Fenchel conjugation to derive a dual problem for a primal one, the convexity, together with the properness and lower semicontinuity of the involved functions in the problem, makes it possible to apply conjugate duality techniques to epigraphs, which ensure strong duality results. Given a primal problem, $(P)$, and its dual, $(D)$, it is said that the pair $(P)-(D)$ verifies \emph{weak duality} when $v(P)\geq v(D)$ and \emph{strong duality} when there is zero duality gap, i.e., $v(P)=v(D)$, and the dual has an optimal solution. Connected to strong duality, \emph{stable strong duality} means that there is strong duality when the objective function of the primal problem is perturbed with a linear functional. For more details on the development of regularity conditions for strong and stable strong dualities, we encourage the reader to revise \cite{B2010,BGW2009,G2016}.

The first application of the $c$-conjugation pattern to duality theory appears in \citep{FVR2012}, building a Fenchel-type dual problem for a primal one and  obtaining conditions ensuring strong duality between them. Later, Lagrange and Fenchel-Lagrange strong dualities were studied in \citep{FVR2016} and \citep{FV2017}, respectively, whereas stable strong duality was developed in \citep{FV2016SSD}. Basically, convexity and lower semicontinuity (closedness, respectively) requirements for functions (for sets, respectively) are replaced by even convexity, a more general and weaker assumption. Evenly convex functions, which will be properly defined in the following section, enjoy a nice behaviour when are $c$-conjugated. A rich theory on optimization and further properties of this class of sets and functions can be found in the recent book \cite{FGRVP2020} and the references therein.

The aim of this paper is to study Fenchel-type duality for a DC primal one, when the dual problems are expressed in terms of the $c$-conjugated of the involved functions in the primal problem. 

The structure of the paper is as follows. First, we will provide in Section 2 the necessary results to make the paper self-contained. Then, Section 3 is devoted to the development of a suitable Fenchel dual problem whenever one of the functions is evenly convex, since weak duality is guaranteed. We also use a convexification technique to obtain an auxiliary dual problem which is related to the primal and the aforementioned dual problem if even convexity is assumed. Section 4 concentrates on characterizations for weak duality and zero duality gap for both dual problems and the primal. In Section 5 we present necessary and sufficient conditions for strong duality between the two dual pairs stated in the previous section, where the $\ep$-convexity (of certain set), a very linked notion to even convexity, will play a relevant role. Section 6 deals with the adaptation of these results to stable strong duality whereas Section 7 concludes the manuscript with a short summary as well as some open ideas left for future research.

\section{Preliminaries}
\label{sec:Pre}

In this section we include the necessary results and definitions used throughout the paper. First, we start denoting by $X$ a nontrivial separated locally convex space, briefly lcs, equipped with the topology induced by the continuous dual space $X^*$, i.e., $\sigma(X,X^*)$. With $\ci x,x^*\cd$ we represent the value of the continuous linear functional $x^*$ at $x\in X$. Given a proper subset $D\subseteq X$, $\conv D$ and $\cl D$ refer to the standard notation of its convex hull and closure, respectively.

In the previous section, we briefly mentioned the class of evenly convex functions as a class which generalizes the family of convex and lower semicontinuous functions defined on $X$. Let us elaborate more on this. Due to \cite{RVP2011}, a function is called \emph{evenly convex}, e-convex in short, if its epigraph is an e-convex set. This means that the epigraph of $f$ is the intersection of an arbitrary family (possibly empty) of open half-spaces. That was the definition that Fenchel used in \cite{F1952} when he applied even convexity to extend the polarity theory to nonclosed convex sets. E-convex sets where studied in \cite{KMZ2007} in terms of their sections and projections. Moreover, \cite{GJR2003, GR2006} introduced them as the solution sets of linear systems containing strict inequalities. Though next result is a characterization of e-convex sets, see \cite{DML2002}, we will use it as its definition, due to its tractability.

\begin{definition}[Def.~1, \cite{DML2002}]
\label{def:Econvex}
A set $C\subseteq X$ is e-convex if for every point $x_0\notin C$, there exists $x^*\in X^*$ such that $\ci x-x_0,x^*\cd<0$, for all $x\in C$.
\end{definition}

Given $C\subseteq X$, the notation $\eco C$ stands for the smallest e-convex set that contains $C$, and it is called its \emph{e-convex hull}. If in addition $C$ is convex, then
\begin{equation*}
	C\subseteq \eco C\subseteq \cl C.
\end{equation*}
This operator is closed under arbitrary intersections and thanks to Hahn-Banach theorem, every closed or open convex set is e-convex as well, but the converse is not necessarily true.

\begin{example}
\label{ex:Econvex_non_closed}
Let us take the set $C$ represented in Figure \ref{fig:Econvex_non_closed} and the point $P$. If we define $D:=C\cup \left\{P\right\}$, observe that the set $D$ is not e-convex since for every point belonging to the dashed edge, every hyperplane containing that point will intersect the set $D$. However, the set $C$ (note, without the point $P$) is e-convex according to Definition \ref{def:Econvex}, but $C$ is neither closed nor open.
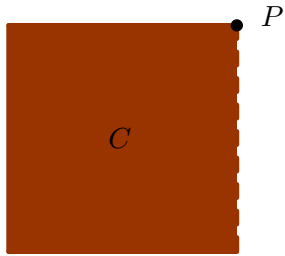
\begin{figure}[h]
\centering
\definecolor{zzttqq}{rgb}{0.6,0.2,0.}
\begin{tikzpicture}[line cap=round,line join=round,>=triangle 45,x=1.0cm,y=1.0cm,scale=0.6]
\clip(-1.25,-0.68) rectangle (6.44,5.96);
\fill[line width=2.pt,color=zzttqq,fill=zzttqq,fill opacity=0.10000000149011612] (0.,0.) -- (0.,5.) -- (5.,5.) -- (5.,0.) -- cycle;
\draw [line width=1.8pt,color=zzttqq] (0.,0.)-- (0.,5.);
\draw [line width=1.8pt,color=zzttqq] (0.,5.)-- (5.,5.);
\draw [line width=1.8pt,dash pattern=on 5pt off 5pt,color=zzttqq] (5.,5.)-- (5.,0.);
\draw [line width=1.8pt,color=zzttqq] (5.,0.)-- (0.,0.);
\draw (5.3,5.64) node[anchor=north west] {$P$};
\draw (1.96,2.96) node[anchor=north west] {$C$};
\begin{scriptsize}
\draw [fill=black] (5.,5.) circle (3.5pt);
\end{scriptsize}
\end{tikzpicture}
\caption{Set in Example \ref{ex:Econvex_non_closed}}
\label{fig:Econvex_non_closed}
\end{figure}
\end{example}

In \cite{RVP2011}, the connection between e-convex sets and functions was naturally done  via the notion of epigraph. Recall that the  \emph{effective domain} and the \emph{epigraph} of a function $f:X\to\Ramp$ are, respectively,
\begin{align*}
	\dom f&:=\left\{ x\in X~:~f(x)<+\infty\right\},\\
	\epi f&:=\left\{(x,\alpha)\in X\times\R~:~f(x)\leq\alpha\right\}
\end{align*}
and it is said that a function is \emph{proper} if $f(x)>-\infty$, for all $x \in X$, and $\dom f\neq \emptyset$. A function is \emph{lower semicontinuous}, lsc in brief, if $f$ coincides with its lower semicontinuous hull, i.e., $f(x)=\cl f(x)$, being the latter the function fulfilling $\epi(\cl f)=\cl(\epi f)$. It is important to notice that if $\mathcal{E}(X)$ and $\Gamma(X)$ represent the families of real extended valued e-convex functions and lsc convex functions defined on the space $X$, respectively, then
\begin{equation*}
	\Gamma(X)\subseteq \mathcal{E}(X),
\end{equation*}
being the converse inclusion false in general as Example 2.1 from \cite{FV2017} shows.

The class of e-convex functions behaves well conjugacy-wise when they are matched with the $c$-conjugation scheme; see \cite{MLVP2011}. This pattern is based on the generalized convex conjugation theory presented by Moreau in \cite{Mor1970}, using the particular coupling functions defined as follows. Let us suppose $f:X\to\Ramp$ is a proper function and denote $W=X^*\times X^*\times \R$. The coupling function used in \emph{$c$-conjugacy} is $c:X\times W\to\Ramp$
\begin{equation*}
	c(x,(x^*,y^*,\alpha))=\left\{
	\begin{aligned}
		\ci x,x^*\cd, &~~~ \text{if } \ci x,y^*\cd<\alpha\\
		+\infty, & ~~~ \text{otherwise.}
	\end{aligned}
	\right.
\end{equation*}
The \emph{$c$-conjugate} of a function $f:X\to\Ramp$ is defined as
\begin{equation*}
	f^c(x^*,y^*,\alpha):=\sup_{X}\left\{ c(x,(x^*,y^*,\alpha))-f(x)\right\}.
\end{equation*}
Let us observe that
\begin{equation*}
	f^c(x^*,y^*,\alpha)= \left \{
	\begin{aligned}
		f^*(x^*), &~~~ \text{if } \dom f \subseteq H_{y^*,\alpha}^{-}\\
		+\infty, & ~~~ \text{otherwise,}
	\end{aligned}
	\right.
\end{equation*}
where $f^*:X^* \to \Ramp$ is the Fenchel conjugate of $f$ and 
$$H_{y^*,\alpha}^{-}:=\{x\in X : \ci x,y^*\cd < \alpha \},$$
using the standard notion (see, for instance, \cite{Z2002}) for open half-spaces in lcs.

This conjugation scheme is complemented by the use of the coupling function $c^\prime:W\times X\to\Ramp$ defined as
\begin{equation*}
	c^\prime((x^*,y^*,\alpha),x):=c(x,(x^*,y^*,\alpha)),
\end{equation*}
which allows the definition of the \emph{$c^\prime$-conjugate} of a proper function $h:W\to\Ramp$ as
\begin{equation*}
	h^{c^\prime}(x):=\sup_{W}\left\{c^{\prime}((x^*,y^*,\alpha),x)-g(x^*,y^*,\alpha)\right\}.
\end{equation*}
The sign convention applied in this conjugation pattern gives priority to $-\infty$, i.e.,
\begin{equation*}
	(+\infty)+(-\infty)=(-\infty)+(+\infty)=(+\infty)-(+\infty)=(-\infty)-(-\infty)=-\infty.
\end{equation*}
It is worth adding that, as it happens for sets, it is possible to define the largest e-convex minorant of $f$, in short its \emph{e-convex hull}, and it is represented by $\eco f$. 

In the setting of generalized convex duality theory, see \cite{Mor1970}, the functions $c(\cdot,(x^*,y^*,\alpha))-\beta :X \rightarrow \Ramp$, with $(x^*,y^*,\alpha) \in W $ and $\beta \in \R$, are called \emph{c-elementary}, and the functions $c^{\prime}(\cdot,x)-\beta : W \rightarrow \Ramp$, with $x\in X$ and $\beta \in \mathbb{R}$, are called \emph{c}$^{\prime }$-\emph{elementary}. In \cite{MLVP2011}, it is shown that any proper e-convex function $f: X \rightarrow \Ramp$ is the pointwise supremum of a set of c-elementary functions. Extending this concept, \cite{FVR2012} introduced the \emph{$\ep$-convex functions} as convex functions $g:W \rightarrow \Ramp$ which can be expressed as the pointwise supremum of a set of c$^\prime$-elementary functions. The \emph{$\ep$-convex hull} of any function $g:W\rightarrow\Ramp$, $\epco g$, is the largest $\ep$-convex function minorant of $g$. The following theorem from \cite{ML2005} is the counterpart of the Fenchel-Moreau theorem for e-convex and $\ep$-convex functions.
 
\begin{theorem}[Prop. 6.1, 6.2, Cor. 6.1,~\cite{ML2005}]
\label{thm:Theorem_charac}
Let $f:X\rightarrow \overline{\mathbb{R}}$ and $g:W\rightarrow \overline{\mathbb{R}}$. Then
\begin{itemize}
\item[(i)] $f^{c}$ is e$^{\prime }$-convex$;$ $g^{c^{\prime }}$ is e-convex.

\item[(ii)] If $f$ has a proper e-convex minorant, $\eco f=f^{cc^{\prime }}$; $ \epco g=g^{c^{\prime }c}$.

\item[(iii)] If $f$ does not take on the value $-\infty $, then $f$ is e-convex if and only if $f=f^{cc^{\prime }};g$ is e$^{\prime }$-convex if and only if $g=g^{c^{\prime }c}$.

\item[(iv)] $f^{cc^{\prime }}\leq f;$ $g^{c^{\prime }c}\leq g$. 
\end{itemize}
\end{theorem}

To close this preliminary section, we add that in \cite{FVR2012} the authors also applied the notion of {e$^\prime$-convexity} to sets $D\subseteq W\times\R$, being those sets such that there exists an e$^\prime$-convex function $h:W\to\R$ verifying that $\epi h=D$.  E$^\prime$-convex sets and functions have been deeply studied and recently characterized in \cite{FV2020}. For completeness, given a set $D\subseteq W\times \R$, the smallest $\ep$-convex set containing $D$ is its \emph{$\ep$-convex hull} and it is represented by $\epco D$. The $\ep$-convex hull of a set has also been used in \cite{F2015,FVR2016,FV2018,FV2016SSD}.

\section{Dual problems}

This section is devoted to build two Fenchel-type dual problems for a given DC optimization problem 
\begin{equation}
	\label{eq:Primal_problem}
	\tag{$P$}
		\begin{aligned}
		&\inf ~ f(x)-g(x)\\
		& \text{ s.t.} ~~~x\in A,
	\end{aligned}
\end{equation}
where $f,g:X\to\Ramp$ are proper convex functions and $A \subseteq X$ is a nonempty set.
As it is usual in DC optimization, we assume as well the sign convention $(+\infty)-(+\infty)=+\infty$ so that the problem $(P)$ is well defined.

We assume that $f-g$ is proper, which means that $\dom f\subseteq \dom g$. In that case, from \cite{FVR2012}, a Fenchel dual problem for $(P)$, fulfilling weak duality, reads
\begin{equation}
	\label{eq:Dual_problem}
	\tag{$D$}
	\begin{aligned}
		\sup_{\substack{x^*,y^*\in X^*\\\alpha_1+\alpha_2>0}} \left\{ -(f-g)^c(-x^*,-y^*,\alpha_1) - \delta_A^c(x^*,y^*,\alpha_2)\right\},
	\end{aligned}
\end{equation}
where $\delta_A$ is the indicator function of the set $A$. Let us observe that, in the calculus of the optimal value of $(D)$, denoted by $v(D)$, we can restrict the analysis to those points $x^*,y^*$ and $\alpha_1+\alpha_2>0$ such that 
\begin{equation*}
	(f-g)^c(-x^*,-y^*,\alpha_1)<+\infty \text{ and } \delta_A^c(x^*,y^*,\alpha_2)<+\infty,
\end{equation*}
which implies that $\dom f \subseteq H^{-}_{-y^*,\alpha_1}$ and $A \subseteq H^{-}_{y^*,\alpha_2}$. For each pair $(\alpha_1,\alpha_2)$ fulfilling both conditions, there exists $\alpha$ big enough (and greater than zero) so that $\dom f \subseteq H^{-}_{-y^*,\alpha}$ and $A \subseteq H^{-}_{y^*,\alpha}$, and hence 
\begin{equation*}
\begin{aligned}
		\sup_{\substack{x^*,y^*\in X^*\\\alpha_1+\alpha_2>0}} & \left\{ -(f-g)^c(-x^*,-y^*,\alpha_1) - \delta_A^c(x^*,y^*,\alpha_2)\right\}\\
		&=\sup_{\substack{x^*,y^*\in X^*, \alpha_1+\alpha_2>0 \\\dom f \subseteq H^{-}_{-y^*,\alpha_1}\\ A \subseteq H^{-}_{y^*,\alpha_2}}} \left\{ -(f-g)^*(-x^*) - \delta_A^*(x^*)\right\}\\
		&=\sup_{\substack{x^*,y^*\in X^*, \alpha>0 \\\dom f \subseteq H^{-}_{-y^*,\alpha}\\ A \subseteq H^{-}_{y^*,\alpha}}} \left\{ -(f-g)^*(-x^*) - \delta_A^*(x^*)\right\} .
	\end{aligned}
\end{equation*}
It allows that $\eqref{eq:Dual_problem}$ could be rewritten as 
\begin{equation*}
	\begin{aligned}
		\sup_{Z} \left\{ -(f-g)^c(-x^*,-y^*,\alpha) - \delta_A^c(x^*,y^*,\alpha)\right\},
	\end{aligned}
\end{equation*}
where $Z:=X^* \times X^* \times \R_{++}$. Recall that, so far, we are just assuming that $f-g$ is proper. If moreover $g$ is e-convex, according to Theorem \ref{thm:Theorem_charac}, we obtain
\begin{equation*}
	(f-g)^c(-x^*,-y^*,\alpha) = \sup_{X} \left\{ c(x,(-x^*,-y^*,\alpha))-f(x)+g^{cc^\prime}(x)\right\}.
\end{equation*}
Applying the definition of $g^{cc^\prime}$, it yields
\begin{align*}
	&(f-g)^c(-x^*,-y^*,\alpha)\\
	&= \sup_{\dom g^c} \sup_X \bigg\{ c(x,(-x^*,-y^*,\alpha)) - [f(x)-c(x,(u^*,v^*,\gamma))] - g^c(u^*,v^*,\gamma) \biggr\},
\end{align*}
or, equivalently,
\begin{align*}
	&(f-g)^c(-x^*,-y^*,\alpha)\\
	& = \sup_{\dom g^c} \left\{ \Bigl(f-c(\cdot,(u^*,v^*,\gamma))\Bigr)^c(-x^*,-y^*,\alpha) - g^c(u^*,v^*,\gamma) \right\}.
\end{align*}
Now, recall that $\dom f\subseteq \dom g$, hence for every $(u^*,v^*,\gamma)\in\dom g^c$, it holds $\dom f\subseteq H_{v^*,\gamma}^-$, so $f-c(\cdot,(u^*,v^*,\gamma)$ will be a proper function, whose effective domain is $\dom f$, allowing us to write
$$\Bigl(f-c(\cdot,(u^*,v^*,\gamma))\Bigr)^c(-x^*,-y^*,\alpha)=f^c(u^*-x^*,-y^*,\alpha).$$
Finally, we have
\begin{equation*}
	(f-g)^c(-x^*,-y^*,\alpha) = \sup_{\dom g^c} \left\{ f^c(u^*-x^*,-y^*,\alpha) - g^c(u^*,v^*,\gamma) \right\},
\end{equation*}
and the formulation of its Fenchel dual problem will be
\begin{equation}
	\label{eq:Dual_problem_final}
	\tag{$D^F$}
	\begin{aligned}
		\sup_{Z} ~\inf_{\dom g^c} \left\{ g^c(u^*,v^*,\gamma)-f^c(u^*-x^*,-y^*,\alpha) - \delta_A^c(x^*,y^*,\alpha)\right\}.
	\end{aligned}
\end{equation}
It is worth mentioning that the supremum in \eqref{eq:Dual_problem_final} is taken on $\dom \delta_A^c$ indeed and, on the other hand, that if the function $g$ is e-convex, weak duality holds by construction, that is $v(P)\geq v(D^F)$. However, as the following example shows, that is not the case in general.

\begin{example}
\label{ex:Example_g_non_econvex_and_WD}
Let $A= [0,+\infty[$ and $f,g:\R\to\Ramp$ defined as
\begin{equation*}
	f(x)=\left\{
	\begin{aligned}
		x,&~~~x\geq 0,\\
		+\infty,&~~~\text{otherwise;}
	\end{aligned}
	\right.
	\hspace{0.75cm}
	\text{and}
	\hspace{0.75cm}
	g(x)=\left\{
	\begin{aligned}
		1,&~~~ x=0,\\
		x,&~~~x> 0,\\
		+\infty,&~~~\text{otherwise.}
	\end{aligned}	
	\right.
\end{equation*}
\begin{center}
\begin{figure}[h]
\centering
\definecolor{ffffff}{rgb}{1.,1.,1.}
\definecolor{ccqqww}{rgb}{0.8,0.,0.4}
\definecolor{zzttqq}{rgb}{0.6,0.2,0.}
\definecolor{uuuuuu}{rgb}{0.26666666666666666,0.26666666666666666,0.26666666666666666}
\begin{tikzpicture}[line cap=round,line join=round,>=triangle 45,x=1.0cm,y=1.0cm,scale=0.75]
\clip(-0.75,-0.94) rectangle (12.2,5.5);
\fill[line width=3.2pt,color=zzttqq,fill=zzttqq,fill opacity=0.10000000149011612] (0.,0.) -- (0.,5.) -- (5.,5.) -- cycle;
\fill[line width=2.pt,color=zzttqq,fill=zzttqq,fill opacity=0.10000000149011612] (7.,2.) -- (7.,5.) -- (12.,5.) -- (7.,0.) -- cycle;
\draw [->,line width=1.pt,dash pattern=on 2pt off 2pt] (0.,0.) -- (0.,5.);
\draw [->,line width=1.pt,dash pattern=on 2pt off 2pt] (0.,0.) -- (5.,0.);
\draw [->,line width=1.pt,dash pattern=on 2pt off 2pt] (7.,0.) -- (7.,5.);
\draw [->,line width=1.pt,dash pattern=on 2pt off 2pt] (7.,0.) -- (12.,0.);
\draw [line width=1.5pt] (0.,0.)-- (5.,5.);
\draw [line width=1.5pt,color=zzttqq] (0.,0.)-- (0.,5.);
\draw [line width=1.5pt,color=zzttqq] (5.,5.)-- (0.,0.);
\draw (0.98,4.) node[anchor=north west] {$epi f$};
\draw [line width=1.5pt,color=zzttqq] (7.,2.)-- (7.,5.);
\draw [line width=1.5pt,color=zzttqq] (12.,5.)-- (7.,0.);
\draw [line width=0.4pt,color=ffffff] (7.,0.)-- (7.,2.);
\draw (8.12,4.04) node[anchor=north west] {$epi g$};
\begin{scriptsize}
\draw [fill=uuuuuu] (0.,0.) circle (2.0pt);
\draw [fill=zzttqq] (7.,2.) circle (2.5pt);
\end{scriptsize}
\end{tikzpicture}
\caption{Epigraphs of the functions $f$ and $g$ in Example \ref{ex:Example_g_non_econvex_and_WD}}
\label{fig:Example_g_noneconvex}
\end{figure}
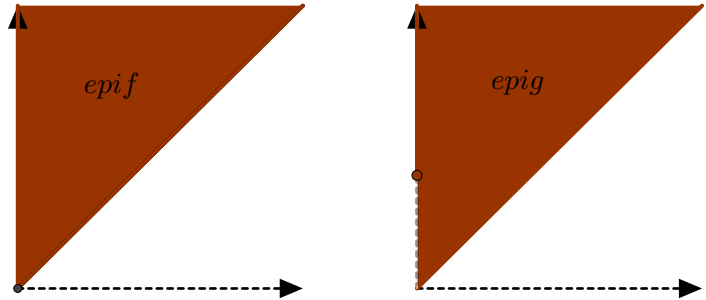
\end{center}
According to Figure \ref{fig:Example_g_noneconvex} and using Definition \ref{def:Econvex}, it is immediate to observe that $f$ is e-convex but $g$ is not. A simple computation shows that $v(P)=-1$. However, if $\alpha>0$, it is not hard to see that 
\begin{align*}
	v \left(D^F \right) &\geq \inf_{\dom g^c} \left\{ g^c(u^*,v^*,\gamma) - f^c(u^*,0,\alpha) - \delta_A^c(0,0,\alpha) \right\}\\
	&= \inf_{\dom g^c} \left\{ 0 - 0 - 0 \right\} =0>-1=v(P).
\end{align*}
\end{example}

\subsection{E-convexification technique to get another dual problem}
\label{subsec:Econvexification}

The purpose of this subsection is to develop another dual problem for $(P)$ whose optimal value will be placed between $v(P)$ and $v \left(D^F \right)$ in the case $g$ is e-convex. Due to the equality $g=g^{cc'}$, assuming that $g$ is e-convex, we will have
\begin{align*}
	v(P)&=\inf_X\left\{ f(x)-\sup_{\dom g^c}\left\{ c(x,(u^*,v^*,\gamma))-g^c(u^*,v^*,\gamma)\right\} +\delta_A(x)\right\}\\
	&= \inf_{\dom g^c} \left\{ \inf_X \left\{ f(x)-c(x,(u^*,v^*,\gamma))+\delta_A(x)\right\} +g^c(u^*,v^*,\gamma)\right\}.
\end{align*}
For every $(u^*,v^*,\gamma)\in\dom g^c$, we define the problem
\begin{equation*}
	\tag{$P^{(u^*,v^*,\gamma)}$}
	\inf_{X} \left\{ f(x)-c(x,(u^*,v^*,\gamma))+\delta_A(x) + g^c(u^*,v^*,\gamma)\right\},
\end{equation*}
which, in the case that the function $f-c(\cdot,(u^*,v^*,\gamma))$ is proper, meaning that $\dom f\subseteq H_{v^*,\gamma}^-$, it is possible to derive a Fenchel dual problem for it (\citep{FVR2012}), reading as
\begin{equation*}
	\tag {$D^{(u^*,v^*,\gamma)}$}
	\sup_{Z} \biggl\{ -\Bigl(f-c(\cdot,(u^*,v^*,\gamma))+g^c(u^*,v^*,\gamma)\Bigr)^c(-x^*,-y^*,\alpha)\\
	 -\delta_A^c(x^*,y^*,\alpha)\biggr\}.
\end{equation*}
Hence, setting $\dom f\subseteq\dom g$, which implies that $\dom f\subseteq H_{v^*,\gamma}^-$ for all $(u^*,v^*,\gamma)\in\dom g^c$, we have
\begin{equation*}
	v(P)=\inf_{\dom g^c} v \left(P^{(u^*,v^*,\gamma)} \right) \geq \inf_{\dom g^c} v\left(D^{(u^*,v^*,\gamma)}\right) = v \left(\Db^F \right),
\end{equation*}
where $ \left(\Db^F \right )$ is the problem
\begin{equation*}
\inf_{\dom g^c} \sup_{Z} \biggl\{-\Bigl(f-c(\cdot,(u^*,v^*,\gamma))+g^c(u^*,v^*,\gamma)\Bigr)^c(-x^*,-y^*,\alpha)\\
	 -\delta_A^c(x^*,y^*,\alpha)\biggr\},
\end{equation*}
fulfilling weak duality. Since we are assuming that $\dom f\subseteq \dom g$, for every $(u^*,v^*,\gamma)\in\dom g^c$ we obtain
\begin{equation*}
\begin{aligned}
	\Bigl(f-c(\cdot,(u^*,v^*,\gamma))&+g^c(u^*,v^*,\gamma)\Bigr)^c(-x^*,-y^*,\alpha)\\
	&= f^c(u^*-x^*,-y^*,\alpha)-g^c(u^*,v^*,\gamma),
\end{aligned}
\end{equation*}
so, equivalently, the dual problem reads as
\begin{equation}\label{eq:dualproblemecon}
	\tag {$\Db^{F}$}
	\inf_{\dom g^c} \sup_{Z} \left\{g^c(u^*,v^*,\gamma)-f^c(u^*-x^*,-y^*,\alpha)-\delta_A^c(x^*,y^*,\alpha)\right\}.
\end{equation}
If $g$ is e-convex, it yields
\begin{equation}
	\label{eq:Optimal_value_ineq}
	v(P)\geq v \left (\Db^F \right ) \geq v\left(D^F \right)
\end{equation}
and, in particular, there exists weak duality between $(P)-\left (\Db^F \right)$. Let us observe that Example \ref{ex:Example_g_non_econvex_and_WD} also serves to show that when the function $g$ is not e-convex, weak duality between $(P)- \left (\Db^F \right )$ is not guaranteed.

\section{Conditions for weak duality and zero duality gap in DC problems}

Example \ref{ex:Example_g_non_econvex_and_WD} shows that when the function $g$ in \eqref{eq:Primal_problem} is not e-convex, it may happen that weak duality fails for the dual pairs $(P)-\left(D^F \right)$ and  $(P)-\left(\Db^F \right)$. For this reason, let us start analysing conditions under which weak duality is guaranteed. We define the sets
\begin{equation}
\label{eq:Set_Omega}
\Omega := \bigcup_{\dom\delta_A^c} \bigcap_{\dom g^c} \Biggl\{ \epi\Bigl(f-c(\cdot,(-x^*,-y^*,\alpha)\Bigr)^c
	-(u^*,0,0,g^c(u^*,v^*,\gamma)-\delta_A^c(x^*,y^*,\alpha))\Biggr\}
\end{equation}
and 
\begin{equation}
	\label{eq:Set_K}
	K:= \bigcap_{\dom g^c}\bigcup_{\dom \delta_A^c} \Biggl\{ \epi\Bigl(f-c(\cdot,(-x^*,-y^*,\alpha)\Bigr)^c
	-(u^*,0,0,g^c(u^*,v^*,\gamma)-\delta_A^c(x^*,y^*,\alpha))
	\Biggr\}.
	\end{equation}

Recalling that $v(P)=\inf_{x\in X} \{f(x)-g(x)+\delta_A(x)\}$, we present, in first place, a very useful lemma, as we will see.
\begin{lemma}
\label{lem:Lemma1}
Let $\beta\in\R$. The following statements hold:
\begin{itemize}
	\item[i)] There exists $\delta>0$ such that $(0,0,\delta,\beta)\in\epi(f-g+\delta_A)^c$ if and only if $v(P)\geq-\beta$.
	\item[ii)] There exists $\delta>0$ such that $(0,0,\delta,\beta)\in\Omega$ if and only there exists $(x^*,y^*,\alpha) \in \dom \delta_A^c$ such that, for all $(u^*,v^*,\gamma)\in\dom g^c$,
	$$g^c(u^*,v^*,\gamma)-f^c(u^*-x^*,-y^*,\alpha)-\delta_A^c(x^*,y^*,\alpha)\geq -\beta.$$
	\item[iii)]  There exists $\delta>0$ such that $(0,0,\delta,\beta)\in K$, if and only if, for all $(x^*,y^*,\alpha) \in \dom \delta_A^c$, there exists $(u^*,v^*,\gamma)\in\dom g^c$ such that
	$$g^c(u^*,v^*,\gamma)-f^c(u^*-x^*,-y^*,\alpha)-\delta_A^c(x^*,y^*,\alpha)\geq -\beta.$$
\end{itemize}
\end{lemma}
\begin{proof}
$i)$ For $\delta>0$, the point $(0,0,\delta,\beta)\in\epi(f-g+\delta_A)^c$ if and only if, for all $x\in X$, the inequality
\begin{equation*}
	c(x,(0,0,\delta))-f(x)+g(x)-\delta_A(x)\leq \beta
\end{equation*}
holds. That is equivalent to claim that, for all $x\in X$, it holds
\begin{equation*}
	f(x)-g(x)+\delta_A(x)\geq -\beta,
\end{equation*}
which means that 
\begin{equation*}
	\inf_{x\in X} \left\{f(x)-g(x)+\delta_A(x)\right\} \geq -\beta,
\end{equation*}
and, hence, $v(P)\geq -\beta$.

$ii)$ From the definition of the set $\Omega$ in (\ref{eq:Set_Omega}), $(0,0,\delta,\beta)\in\Omega$, $\delta >0$, if and only if there exists $(x^*,y^*,\alpha) \in \dom \delta_A^c$ such that, for all $(u^*,v^*,\gamma)\in\dom g^c$, 
$$(0,0,\delta,\beta)+(u^*,0,0,g^c(u^*,v^*,\gamma)-\delta_A^c(x^*,y^*,\alpha)) \in \epi \Big(f-c(\cdot,(-x^*,-y^*,\alpha))\Bigr)^c.$$
This is equivalent to the existence of $(x^*,y^*,\alpha) \in \dom \delta_A^c$ such that, for all $(u^*,v^*,\gamma)\in\dom g^c$,
$$\Bigl(f-c(\cdot,(-x^*,-y^*,\alpha))\Bigr)^c(u^*, 0, \delta) \leq \beta + g^c(u^*,v^*,\gamma)-\delta_A^c(x^*,y^*,\alpha),$$
which can also be written as
\begin{equation}\label{eq:ineqlema41}
	f^c(u^*-x^*,-y^*,\alpha)\leq \beta +g^c(u^*,v^*,\gamma)-\delta_A^c(x^*,y^*,\alpha).
\end{equation}
Then $(0,0,\delta,\beta)\in\Omega$, $\delta >0$,  if and only if there exists $(x^*,y^*,\alpha) \in \dom \delta_A^c$ such that, for all $(u^*,v^*,\gamma)\in\dom g^c$,  
\begin{equation*}
	g^c(u^*,v^*,\gamma)-f^c(u^*-x^*,-y^*,\alpha)-\delta_A^c(x^*,y^*\alpha)\geq -\beta.
\end{equation*}
 The proof of $iii)$ is similar.
\end{proof}


\begin{remark}\label{remark:forall}
In the  previous lemma, it is easy to see that the existence of $\delta >0$ in  $i), ii)$ and $iii)$ verifying the required condition in each case, is equivalent to say that for all $\delta >0$ the required condition is fulfilled. 
\end{remark}
Next proposition shows necessary and sufficient conditions for weak duality in a general case. Before we state it, we introduce the set
\begin{equation}
	\label{eq:Set_B}
	B:=\left\{ (0,0)\times \R_{++}\times \R\right\} \subseteq W \times \R.
\end{equation}
\begin{proposition}
\label{prop:NC_for_WD}
Let $\Omega$, $K$ and $B$ be the sets defined in \eqref{eq:Set_Omega}, \eqref{eq:Set_K} and \eqref{eq:Set_B}, respectively. Then:
\begin{itemize}
\item [i)]$v(P)\geq v\left(D^F \right)$ if and only if $\Omega\cap B \subseteq \epi(f-g+\delta_A)^c\cap B$.

\item[ii)] $v(P)\geq v \left(\Db^F \right)$ if and only if $K\cap B\subseteq \epi(f-g+\delta_A)^c\cap B$.
\end{itemize}

\end{proposition}
\begin{proof}
$i)$ Let us assume that $v(P)\geq v\left(D^F \right)$ and take $(0,0,\delta,\beta)\in \Omega$ with $\delta>0$. From Lemma \ref{lem:Lemma1} $ii)$, there exists $(x^*,y^*,\alpha) \in \dom \delta_A^c$ such that, for all $(u^*,v^*,\gamma)\in\dom g^c$,
	$$g^c(u^*,v^*,\gamma)-f^c(u^*-x^*,-y^*,\alpha)-\delta_A^c(x^*,y^*,\alpha)\geq -\beta,$$
and then $v\left(D^F \right)\geq-\beta$, so $v(P)\geq -\beta$. Finally, Lemma \ref{lem:Lemma1} $i)$ implies that $(0,0,\delta,\beta)\in\epi(f-g+\delta_A)^c$.

To prove the converse implication, we proceed by contradiction. If $v(P)<v\left(D^F \right)$, we can take $\beta\in\R$ such that
\begin{equation*}
	v(P)<-\beta< v\left(D^F \right).
\end{equation*}
Hence, since $v\left(D^F \right)> -\beta$, there exists $(x^*,y^*,\alpha) \in \dom \delta_A^c$ such that, for all $(u^*,v^*,\gamma)\in\dom g^c$, 
	$$g^c(u^*,v^*,\gamma)-f^c(u^*-x^*,-y^*,\alpha)-\delta_A^c(x^*,y^*,\alpha)\geq -\beta,$$
and, by Lemma \ref{lem:Lemma1} $ii)$, $(0,0,\delta,\beta)\in \Omega$, for all $\delta>0$ and, whence, $(0,0,\delta,\beta)\in\epi(f-g+\delta_A)^c$ by hypothesis. Finally,  Lemma \ref{lem:Lemma1} $i)$ implies that $v(P)\geq -\beta$, which gives us the expected contradiction which concludes the proof of $i)$. The proof of $ii)$ follows similar steps.
\end{proof}


We finish this section with conditions which are equivalent to zero duality gap for $(P)-\left(D^F \right)$ and $(P)-\left(\Db^F \right )$.

\begin{proposition}\label{prop:ZDG}
Let $\Omega$, $K$ and $B$ be the sets defined in \eqref{eq:Set_Omega}, \eqref{eq:Set_K} and \eqref{eq:Set_B}, respectively. Then:
\begin{itemize}
	\item [i)]$v(P)=v\left(D^F \right)$ if and only if $\epco(\Omega\cap B)= \epi(f-g+\delta_A)^c\cap B.$
	\item[ii)] $v(P)= v\left(\Db^F \right)$ if and only if $\epco(K \cap B)=\epi(f-g+\delta_A)^c\cap B.$ 
\end{itemize}
\end{proposition}

\begin{proof}
It will be sufficient to prove $i)$ because the proof of $ii)$ follows the same steps.
Suppose that $v(P)=v\left(D^F \right)$. According to Proposition \ref{prop:NC_for_WD} $i)$, $$\Omega\cap B\subseteq \epi(f-g+\delta_A)^c\cap B.$$ 
Now, take a point $(0,0,\delta, \beta) \in \epi(f-g+\delta_A)^c$, with $\delta >0$. By Lemma \ref{lem:Lemma1} $i)$ (recall Remark \ref{remark:forall}), $v(P)\geq -\beta$  and by hypothesis, $v\left(D^F \right)\geq -\beta$, so, for all $\varepsilon >\beta$,  $v\left(D^F \right)> -\varepsilon$ and hence, for each $\varepsilon$, there exists $({x_\varepsilon}^*,{y_\varepsilon}^*,\alpha_\varepsilon) \in \dom \delta_A^c$, such that, for all $(u^*,v^*,\gamma)\in\dom g^c$,
\begin{equation}\label{eq:varepsilon}
g^c(u^*,v^*,\gamma)-f^c(u^*-{x_\varepsilon}^*,-{y_\varepsilon}^*,\alpha_\varepsilon)-\delta_A^c({x_\varepsilon}^*,{y_\varepsilon}^*,\alpha_\varepsilon)\geq -\varepsilon.
\end{equation}
According to Lemma \ref{lem:Lemma1} $ii)$, $(0,0,\delta, \varepsilon) \in \Omega \cap B$, for all $ \varepsilon >\beta$.

Let us assume that $\epco(\Omega\cap B)=\epi H$, where $H:W\times \R \to \Ramp$ is an $\ep$-convex function, and
$$H=\sup_{(x,\lambda) \in D} c^\prime(\cdot,x)-\lambda,$$
with $D \subseteq X \times \R$. If $(0,0,\delta, \beta) \notin \epco(\Omega\cap B)$, there would exist $(\bar x,\bar \lambda) \in D$ such that $(0,0,\delta, \beta) \notin \epi c^\prime(\cdot,\bar x)-\bar \lambda$, and $c^\prime((0,0,\delta),\bar x)-\bar \lambda > \beta$, equivalently,
\begin{equation} \label{eq:lesstric}
\bar \lambda < -\beta.
\end{equation}
On the other hand, by \eqref{eq:varepsilon} and Lemma \ref{lem:Lemma1} $ii)$, $(0,0,\delta, \varepsilon)\in \Omega\cap B$, then  $(0,0,\delta, \varepsilon) \subseteq \epco(\Omega\cap B)$, for all $ \varepsilon >\beta$, and $c^\prime(\bar x,(0,0,\delta))-\bar \lambda \leq \varepsilon$, for all $ \varepsilon >\beta$, equivalently,
$$\bar \lambda \geq -\varepsilon,$$
for all $ \varepsilon >\beta$, contradicting (\ref{eq:lesstric}). It follows that
$$(\Omega\cap B)\subseteq \epi(f-g+\delta_A)^c\cap B \subseteq \epco(\Omega\cap B).$$ 
 If we prove that $\epi(f-g+\delta_A)^c\cap B $ is $\ep$-convex, since $\epco(\Omega\cap B)$ is the smallest $\ep$-convex set containing $\Omega\cap B$,  it will be 
$$\epi(f-g+\delta_A)^c\cap B = \epco(\Omega\cap B).$$
 From Lemma \ref{lem:Lemma1} $i)$, $(0,0,\delta, \beta) \in \epi(f-g+\delta_A)^c \cap B$ if and only if $v(P) \geq -\beta$ and $\delta >0$, which  means that 
 \begin{equation}\label{eq:epifequality}
 \epi(f-g+\delta_A)^c \cap B=\{(0,0,\delta, \beta): \delta>0, \beta \geq -v(P)\}.
 \end{equation}
We can write
$$\epi(f-g+\delta_A)^c \cap B= \bigcap_{x\in \R} \epi c^{\prime} (\cdot,x)-v(P),$$
and hence $\epi(f-g+\delta_A)^c \cap B$ is an $\ep$-convex set.\\
Conversely, let us assume that $$\epco(\Omega\cap B)=\epi(f-g+\delta_A)^c\cap B.$$
Then, by (\ref{eq:epifequality}), 
 \begin{equation}\label{eq:eponvequality}
 \epco(\Omega\cap B)=\{(0,0,\delta, \beta): \delta>0, \beta \geq -v(P)\}.
 \end{equation}
Let us show that $v(P)=v(D^F)$. Since $\Omega\cap B\subseteq \epi(f-g+\delta_A)^c\cap B$, by Proposition \ref{prop:NC_for_WD} $i)$, $v(P)\geq v(D^F)$. Now, on the contrary, suppose that $v(P)> v\left(D^F \right)$, and take $\bar \beta\in\R$ such that 
\begin{equation} 
\label{eq:strict}
v(P)> -\bar \beta > v\left(D^F \right).
\end{equation}
Take a fix real number $\bar \delta >0$ and let us consider the point $(0,0, \bar\delta, \bar\beta) \in \epco(\Omega\cap B)$, according to (\ref{eq:eponvequality}). In the case we could not find a net $(0,0,\delta_\tau, \beta_\tau)_{\tau \in T} \subset \Omega \cap B$ converging to $(0,0,\bar \delta, \bar\beta)$, it would happen that, for all $(0,0,\delta,\beta) \in \Omega \cap B$, $\beta \geq \bar \beta + \varepsilon$, for some $\varepsilon >0$. Then
$$\Omega \cap B \subseteq \{(0,0,\delta, \beta): \delta>0, \beta \geq \bar \beta + \varepsilon\}=\bigcap_{x\in \R} \epi c^{\prime} (\cdot,x)+(\bar \beta +\varepsilon) \subsetneq \epco(\Omega\cap B),$$
and we would have found an $ \ep$-convex set containing $\Omega \cap B$ and strictly contained in its $\ep$-convex hull, which is not possible.\\
Take then a net $(0,0,\delta_\tau, \beta_\tau)_{\tau \in T} \subset \Omega \cap B$ converging to  $(0,0, \bar \delta, \bar\beta)$. We will have, by Lemma  \ref{lem:Lemma1} $ii)$, for each $\tau \in T$, a point $({x_\tau}^*,{y_\tau}^*,\alpha_\tau) \in \dom \delta_A^c$ such that, for all $(u^*,v^*,\gamma)\in\dom g^c$,
$$g^c(u^*,v^*,\gamma)-f^c(u^*-{x_\tau}^*,-{y_\tau}^*,\alpha_\tau)-\delta_A^c({x_\tau}^*,{y_\tau}^*,\alpha_\tau)\geq -\beta_\tau,$$
hence $v\left(D^F \right)\geq -\beta_\tau$, for all $\tau \in T$, and $v\left(D^F \right)\geq -\beta$, contradicting \eqref{eq:strict}.

\end{proof}

\section{Conditions for strong duality in DC problems}
After characterizing weak duality and zero duality gap for the dual pairs $(P)-(D^F)$ and $(P)-\left(\Db^F \right)$, we continue investigating further conditions for strong duality.

\begin{proposition}
\label{prop:Strong_duality_via_Omega}
Let $\Omega$, $K$ and $B$ be the sets defined by \eqref{eq:Set_Omega}, \eqref{eq:Set_K} and \eqref{eq:Set_B}, respectively. Then:
\begin{itemize}
\item[i)] Strong duality holds  for $(P)-(D^F)$ if and only if $$\Omega\cap B  = \epi(f-g+\delta_A)^c\cap B.$$

\item[ii)] Strong duality holds for $(P)-\left(\Db^F \right )$ if and only if 
	\begin{equation*}
		K\cap B=\epi(f-g+\delta_A)^c\cap B.
	\end{equation*}
\end{itemize}

\end{proposition}
\begin{proof}
$i)$ Let us assume that there exists strong duality between $(P)$ and $(D^F)$. By Proposition \ref{prop:NC_for_WD} $i)$, we have the inclusion
\begin{equation*}
	\Omega\cap B  \subseteq \epi(f-g+\delta_A)^c\cap B.
\end{equation*}
To check the reverse inclusion, take $(0,0,\delta,\beta)\in\epi(f-g+\delta_A)^c\cap B$, with $\delta >0$. Then, by Lemma \ref{lem:Lemma1} $i)$, $v(P)\geq-\beta$. Since strong duality holds by hypothesis, there exists $(x_0^*,y_0^*,\alpha_0)\in\dom\delta_A^c$ such that, for all $(u^*,v^*,\gamma)\in\dom g^c$
\begin{equation*}
	v(D^F) = g^c(u^*,v^*,\gamma)-f^c(u^*-x_0^*,-y_0^*,\alpha_0) +\delta_A^c(x_0^*,y_0^*,\alpha_0)\geq-\beta.
\end{equation*}
According to Lemma \ref{lem:Lemma1} $ii)$, it yields $(0,0,\delta,\beta)\in\Omega\cap B$.

We continue showing the converse implication. By Proposition \ref{prop:NC_for_WD} $i)$, $v(D^F)\leq v(P)$. In case $v(P)=-\infty$, strong duality holds trivially, so let us assume that $v(P)=-\beta\in\R$. \\
It was shown in Proposition \ref{prop:ZDG} that, in general, $\epi(f-g+\delta_A)^c\cap B$ is an $\ep$-convex set, since we can express it as
$$\epi(f-g+\delta_A)^c \cap B= \bigcap_{x\in \R} \epi c^{\prime} (\cdot,x)-v(P).$$
According to the hypothesis we conclude that $\Omega \cap B$ is, then, an $\ep$-convex set, and, hence, by Proposition \ref{prop:ZDG} $i)$, $-\beta=v(P)=v(D^F)$.\\
By Lemma \ref{lem:Lemma1} $i)$, $(0,0,\delta,\beta) \in \epi(f-g+\delta_A)^c \cap B$, for any $\delta >0$, and again, due to the equality in the hypothesis, and Lemma \ref{lem:Lemma1} $ii)$, 
there must exist $\xb^*,\yb^*\in X^*$ and $\alphab>0$ such that for all $(u^*,v^*,\gamma)\in\dom g^c$, it holds
\begin{equation*}
	 g^c(u^*,v^*,\gamma)-f^c(u^*-\xb^*,-\yb^*,\alphab) - \delta_A^c(\xb^*,\yb^*,\alphab)\geq -\beta=v(D^F),
\end{equation*}
i.e., $v(D^F)$ is solvable and there exists strong duality for the dual pair $(P)-(D^F)$. The proof of $ii)$ follows similar steps.
\end{proof}

Next example shows, as it is predictable, that
$$\Omega \cap B \subsetneq\epco(\Omega\cap B)= \epi(f-g+\delta_A)^c\cap B$$
in a DC problem when there is zero duality gap and the dual problem $(D^F)$ is non-solvable, see Propositions \ref{prop:ZDG} and \ref{prop:Strong_duality_via_Omega}.
\begin{example}	\label{ex:Strict_containments}
Let us define $f,g:\R \to \Ramp$,
\begin{equation*}
	f(x)=\left\{
	\begin{aligned}
		-2\sqrt {-x}, &~~~x\leq 0,\\
		+\infty, &~~~ \text{otherwise,}
	\end{aligned}
	\right.
	\hspace{1cm}\text{and}\hspace{1cm}
	g(x)=\left\{
	\begin{aligned}
		-x, &~~~x \leq 0,\\
		+\infty, &~~~ \text{otherwise,}
	\end{aligned}
	\right.
\end{equation*}
with $A=[0, +\infty[$. Applying the definition of the $c$-conjugate function, and denoting by $D:=(\R_+ \times \R_+)\setminus\{(0,0)\}$,
\begin{equation*}
	f^c(x^*,y^*,\alpha)=\left\{
	\begin{aligned}
		{1 \over {x^*}}, &~~~ \text{if }x^*> 0,\,(y^*,\alpha) \in D,\\
		+\infty, &~~~ \text{otherwise,}
	\end{aligned}
	\right.
\end{equation*}
whereas
\begin{equation*}
	g^c(u^*,v^*,\gamma)=\left\{
	\begin{aligned}
		0,&~~~ \text{if } u^*\geq -1,\,(v^*,\gamma) \in D,\\
		+\infty, &~~ \text{ otherwise,}
	\end{aligned}
	\right.
\end{equation*}
and
\begin{equation*}
	\delta_A^c(x^*,y^*,\alpha)=\left\{
	\begin{aligned}
		0,&~~~ \text{if } x^*\leq 0,\,(y^*,\alpha) \in D,\\
		+\infty, &~~ \text{ otherwise.}
	\end{aligned}
	\right.
\end{equation*}
It is evident that $v(P)=0$.
Now, 

\begin{equation*}
v(D^F)={\sup_{\substack{ x^* \leq 0,\\(y^*,\alpha) \in D} }}{\inf_{\substack{ u^* > x^*,\\u^* \geq -1,\\(v^*,\gamma) \in D}}} {1\over {x^*-u^*}}={\sup_{\substack{ x^* \leq 0,\\(y^*,\alpha) \in D} }}{1\over {x^*-1}}=0,
\end{equation*}
and it is non-solvable.\\
Finally, let us compute the sets $(\Omega\cap B)$ and  $\epi(f-g+\delta_A)^c\cap B$. It is clear that, since $\dom \delta_A \cap \dom f= \{0 \}$, we will have that  

$$\epi(f-g+\delta_A)^c=\R \times \R \times \R_{++} \times \R_+,$$
and 
$$\epi(f-g+\delta_A)^c\cap B =\{(0,0)\} \times \R_{++} \times \R_+.$$

\noindent
In Proposition \ref{prop:Econv_formula}, the set $\Omega$ will be characterized as it follows, in the case $g$ is e-convex,
$$\Omega = \bigcup_{\dom\delta_A^c} \epi\Bigl(f- g - c(\cdot,(-x^*,-y^*,\alpha))-\delta_A^c(x^*,y^*,\alpha)\Bigr)^c.$$

\noindent
Take $(x^*, y^*, \alpha)\in \dom \delta^c_A=\R_{-}\times D$. Then, a simple computation leads to
\begin{align*}
	(f- g - c(\cdot,(-x^*,-y^*,\alpha))\Bigr)^c & (u^*, v^*, \gamma)\\
	& =\left\{
	\begin{aligned}
		{1 \over {u^*+x^*+1}}, &~~~ \text{if }u^*> -x^*-1,\,(v^*,\gamma) \in D,\\
		+\infty, &~~~ \text{otherwise,}
	\end{aligned}
	\right.
\end{align*}
and 
\begin{equation*}
	\epi\Bigl(f- g - c(\cdot,(-x^*,-y^*,\alpha))\Bigr)^c=]-x^*-1,+\infty[\,\times D\times \Big[ {1 \over {u^*+x^*+1}}, +\infty\Big[.
\end{equation*}

\noindent
Hence
$$\Omega = \bigcup_{\R_{-}} ~ ]-x^*-1,+\infty[\,\times D\times \Big[{1 \over {u^*+x^*+1}}, +\infty\Big[,$$
and 
$$\Omega \cap B=\{(0,0)\} \times \R_{++} \times \R_{++}.$$

\noindent
Figure \ref{fig:Strict_containments} contains the projection of the two computed sets onto $\R^2$ (recall that these sets are contained in $X^* \times X^* \times \R^2$).
\begin{center}
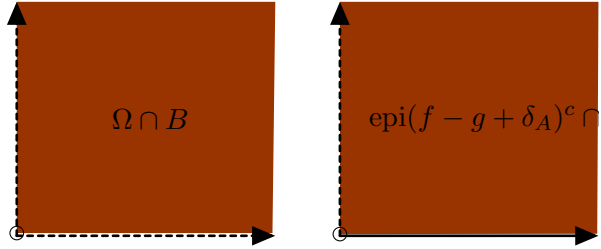
\begin{figure}[h]
\centering
\definecolor{zzttqq}{rgb}{0.6,0.2,0.}
\begin{tikzpicture}[line cap=round,line join=round,>=triangle 45,x=1.0cm,y=1.0cm,scale=0.85]
\clip(-2,1) rectangle (8,5);
\fill[line width=2.pt,color=zzttqq,fill=zzttqq,fill opacity=0.10000000149011612] (-1.,5.) -- (-1.,1.3927272727272735) -- (2.956198347107435,1.3927272727272737) -- (3.,5.) -- cycle;
\fill[line width=2.pt,color=zzttqq,fill=zzttqq,fill opacity=0.10000000149011612] (4.,5.) -- (4.,1.3761983471074388) -- (7.980991735537183,1.3817079889807171) -- (8.,5.) -- cycle;
\draw [->,line width=1.pt,dash pattern=on 2pt off 2pt] (-1.,1.35) -- (-1.,5.);
\draw [->,line width=1.pt,dash pattern=on 2pt off 2pt] (-1.,1.35) -- (3.,1.35);
\draw [->,line width=1.pt,dash pattern=on 2pt off 2pt] (4.,1.35) -- (4.,5.);
\draw [->,line width=1.pt] (4.,1.35) -- (8.,1.35);
\draw (0.31157024793388177,3.480881542699726) node[anchor=north west] {$\Omega\cap B$};
\draw (4.2917355371900795,3.5635261707989) node[anchor=north west] {$\text{epi}(f-g+\delta_A)^c \cap B$};
\begin{scriptsize}
\draw [color=black] (-1.,1.3927272727272735) circle (3.0pt);
\draw [color=black] (4.,1.3761983471074388) circle (3.0pt);
\end{scriptsize}
\end{tikzpicture}
\caption{Projections onto $\R^2$ in Example \ref{ex:Strict_containments}}
\label{fig:Strict_containments}
\end{figure}
\end{center}
\end{example}
Next corollary establishes that $\ep$-convexity of the set $\Omega\cap B$ or $K\cap B$  is a necessary requirement for strong duality in each dual pair. Its proof is inmediate from Propositions \ref{prop:ZDG} and \ref{prop:Strong_duality_via_Omega}.

\begin{corollary}
Let $\Omega$, $K$ and $B$ be the sets defined by \eqref{eq:Set_Omega}, \eqref{eq:Set_K} and \eqref{eq:Set_B}, respectively. Then strong duality holds  for $(P)-(D^F)$ $((P)-\left(\Db^F \right )$, respectively$)$ if $\Omega\cap B$ $(K\cap B$, respectively$)$ is $\ep$-convex. 
\end{corollary}

Our aim now is to show in which way strong duality for $(P)-(D^F)$ and $(P)-\left (\Db^F \right )$ are related. We need some preliminary results. The first one is a characterization for strong duality between the dual pair
\begin{equation*}
	\begin{aligned}
	(P_0)&~~~\inf_{X} \Bigl\{ f(x)+\delta_A(x) \Bigr\}\\
	(D_0)&~~~\sup_{\dom \delta_A^c} \Bigl\{ -f^c(-x^*,-y^*,\alpha)-\delta_A^c(x^*,y^*,\alpha)\Bigr\},
	\end{aligned}
\end{equation*}
that is, the particular case of a DC problem and its dual Fenchel with $g\equiv 0$.  Strong duality conditions for this pair of dual problems were studied firstly in \cite{FVR2012}, where it is shown that, by construction, weak duality holds, i.e., $v(P_0)\geq v(D_0)$.
\begin{proposition}
\label{thm:SD_for_P0_D0}
Let $B$ be the set defined in \eqref{eq:Set_B}. There exists strong duality for $(P_0)-(D_0)$ if and only if
\begin{equation*}
	\epi(f+\delta_A)^c\cap B =\bigcup_{\dom \delta_A^c}\Bigl \{ \epi\Bigl(f-c(\cdot,(-x^*,-y^*,\alpha))-\delta_A^c(x^*,y^*,\alpha)\Bigr)^c\cap B \Bigr \}.
\end{equation*}
\end{proposition}
\begin{proof}
Let us start proving that there exists strong duality between $(P_0)-(D_0)$. If $v(P_0)=-\infty$, strong duality holds trivially, since weak duality fulfills, so let $v(P_0)=-\beta\in\R$. Then, according to Lemma \ref{lem:Lemma1} $i)$, $(0,0,\delta,\beta)\in\epi(f+\delta_A)^c$ for all $\delta>0$. By hypothesis, there exists $(\xb^*,\yb^*,\alphab)\in\dom\delta_A^c$ such that
\begin{equation*}
	c(x,(0,0,\delta))-f(x)+c(x,(-\xb^*,-\yb^*,\alphab))+\delta_A^c(\xb^*,\yb^*,\alphab)\leq \beta
\end{equation*}
for all $x\in X$, which is equivalent to say that
\begin{equation*}
	-f^c(-\xb^*,-\yb^*,\alphab)-\delta_A^c(\xb^*,\yb^*,\alphab)\geq -\beta,
\end{equation*}
or, in other words,
\begin{equation*}
	-\beta=v(P_0)\geq v(D_0) \geq -f^c(-\xb^*,-\yb^*,\alphab)-\delta_A^c(\xb^*,\yb^*,\alphab)\geq -\beta,
\end{equation*}
and, then, strong duality holds for $(P_0)-(D_0)$.\\
Conversely, suppose that $v(P_0)=v(D_0)$ and the dual problem is solvable.
Take $(\bar x^*,\bar y^*,\bar \alpha) \in \dom \delta_A^c$ and  
\begin{equation*}
	(0,0,\delta,\beta) \in \epi\Bigl(f-c(\cdot,(-\bar x^*,-\bar y^*,\bar \alpha))-\delta_A^c(\bar x^*,\bar y^*,\bar \alpha)\Bigr)^c.
\end{equation*}
Then, reasoning as before, we have $v(D_0)\geq -\beta$ and also $v(P_0)\geq -\beta$. By Lemma \ref{lem:Lemma1} $i)$, $(0,0,\delta,\beta)\in\epi(f+\delta_A)^c$ and 
\begin{equation*}
\epi(f+\delta_A)^c\cap B \supseteq \bigcup_{\dom \delta_A^c}\Bigl \{ \epi\Bigl(f-c(\cdot,(-x^*,-y^*,\alpha))-\delta_A^c(x^*,y^*,\alpha)\Bigr)^c\cap B \Bigr \}.
\end{equation*}
Finally, let us show the converse inclusion. Take $(0,0,\delta,\beta)\in\epi(f+\delta_A)^c$ with $\delta>0$. Then, due to Lemma \ref{lem:Lemma1} $i)$, it holds $v(P_0)\geq -\beta$. By hypothesis, $v(D_0)=v(P_0)$ and there exists an optimal solution $(\xb^*,\yb^*,\alphab)\in\dom\delta_A^c$ such that
\begin{equation*}	
	v(D_0)=-f^c(-\xb^*,-\yb^*,\alphab)-\delta_A^c(\xb^*,\yb^*,\alphab)\geq-\beta.
\end{equation*}
Then, applying the definition of the $c$-conjugate function
\begin{equation*}
	c(x,(-\xb^*,-\yb^*,\alphab))-f(x)+\delta_A^c(\xb^*,\yb^*,\alphab)\leq \beta
\end{equation*}
for all $x\in X$, and 
\begin{equation*}
	c(x,(0,0,\delta))-[f-c(\cdot,(-\xb^*,-\yb^*,\alphab))-\delta_A^c(\xb^*,\yb^*,\alphab)](x)\leq \beta
\end{equation*}
for all $x\in X$. Then,
\begin{equation*}
	\Bigl(f-c(\cdot,(-\xb^*,-\yb^*,\alphab))-\delta_A^c(\xb^*,\yb^*,\alphab)\Bigr)^c(0,0,\delta)\leq \beta,
\end{equation*}
which means that
\begin{equation*}
	(0,0,\delta,\beta)\in\epi\Bigl(f-c(\cdot,(-\xb^*,-\yb^*,\alphab))-\delta_A^c(\xb^*,\yb^*,\alphab)\Bigr)^c
\end{equation*}
concluding the proof.
\end{proof}

In this moment, we need to define the set
\begin{equation}
	\label{eq:Set_Xi}
	\Lambda := \bigcap_{\dom g^c} \biggl\{ \epi (f+\delta_A)^c - (u^*,0,0,g^c(u^*,v^*,\gamma)) \biggr\}.
\end{equation}
We relate this set and $\Omega$ defined in \eqref{eq:Set_Omega} with the function $f-\eco g$ in the next result.

\begin{proposition}
\label{prop:Econv_formula}
Let $\Omega$ and $\Lambda$ the sets defined in \eqref{eq:Set_Omega} and \eqref{eq:Set_Xi}, respectively. Then, the following equalities hold:
\begin{equation*}
	\begin{aligned}
	i)~~\Lambda &= \epi(f-\eco g+\delta_A)^c;\\
	ii)~~\Omega &= \bigcup_{\dom\delta_A^c} \epi\Bigl(f-\eco g - c(\cdot,(-x^*,-y^*,\alpha))-\delta_A^c(x^*,y^*,\alpha)\Bigr)^c.
	\end{aligned}
\end{equation*}
\end{proposition}
\begin{proof}
$i)$ Due to basic properties of the conjugate function, we observe that
\begin{align*}
	(f-\eco g+\delta_A)^c &= \biggl(f+\delta_A-\sup_{\dom g^c}\left\{ c(\cdot,(u^*,v^*,\gamma))-g^c(u^*,v^*,\gamma)\right\}\biggr)^c\\
	&= \left(\inf_{\dom g^c} \biggl\{ f+\delta_A+g^c(u^*,v^*,\gamma)-c(\cdot,(u^*,v^*,\gamma))\biggr\} \right)^c\\
	&= \sup_{\dom g^c} \biggl\{ f+\delta_A-c(\cdot,(u^*,v^*,\gamma))+g^c(u^*,v^*,\gamma)\biggr\}^c.
\end{align*}
Hence, from this chain of equalities, we obtain
\begin{equation*}
	\epi(f-\eco g+\delta_A)^c = \bigcap_{\dom g^c} \epi\Bigl(f+\delta_A-c(\cdot,(u^*,v^*,\gamma))+g^c(u^*,v^*,\gamma)\Bigr)^c.
\end{equation*}
So, if $(w^*,z^*,\lambda,\beta)\in\epi\Bigl(f+\delta_A-c(\cdot,(u^*,v^*,\gamma)) + g^c(u^*,v^*,\gamma)	\Bigr)^c$, for any point $(u^*,v^*,\gamma) \in \dom g^c$, using the fact that $\dom f \subseteq \dom g \subseteq H_{v^*,\gamma}^-$, we conclude that $(f+\delta_A)^c(w^*+u^*,z^*,\lambda)\leq \beta+g^c(u^*,v^*,\gamma)$,
which means that
\begin{equation*}
	(w^*+u^*,z^*,\lambda,\beta+g^c(u^*,v^*,\gamma)) \in \epi(f+\delta_A)^c
\end{equation*}
or, equivalently, $\epi(f-\eco g+\delta_A)^c=\Lambda$.

$ii)$ In a similar way as in $i)$, applying basic properties of $c$-conjugation theory, for any point $(x^*,y^*,\alpha) \in \dom\delta_A^c$, one gets
\begin{equation*}
\begin{aligned}
	&\Bigl(f-c(\cdot,(-x^*,-y^*,\alpha)) -\eco g-\delta_A^c(x^*,y^*,\alpha)\Bigr)^c\\
	&=\sup_{\dom g^c} \biggl\{ f-c(\cdot,(-x^*,-y^*,\alpha))-c(\cdot,(u^*,v^*,\gamma))+g^c(u^*,v^*,\gamma)-\delta_A^c(x^*,y^*,\alpha)\biggr\}^c,
\end{aligned}
\end{equation*}
and we obtain the equality
\begin{align*}
	& \epi\Bigl(f-c(\cdot,(-x^*,-y^*,\alpha))-\eco g-\delta_A^c(x^*,y^*,\alpha)\Bigr)^c\\
	&= \bigcap_{\dom g^c} \epi\Bigl(f-c(\cdot,(-x^*,-y^*,\alpha))-c(\cdot,(u^*,v^*,\gamma))+ g^c(u^*,v^*,\gamma)-\delta_A^c(x^*,y^*,\alpha)\Bigr)^c.
\end{align*}
In this way, if $(w^*,z^*,\lambda,\beta)$ belongs to
\begin{equation*}
	\epi\Bigl(f-c(\cdot,(-x^*,-y^*,\alpha))-c(\cdot,(u^*,v^*,\gamma)) + g^c(u^*,v^*,\gamma)-\delta_A^c(x^*,y^*,\alpha)\Bigr)^c,
\end{equation*}
since $\dom f\subseteq {H_{v^*,\gamma}^-}$, we have, equivalently, that 
\begin{equation*}
	(w^*+u^*,z^*,\lambda,\beta+g^c(u^*,v^*,\gamma)-\delta_A^c(x^*,y^*,\alpha))\in\epi\Bigl(f-c(\cdot,(-x^*,-y^*,\alpha))\Bigr)^c,
\end{equation*}
and the proof of $ii)$ is completed according to the definition of $\Omega$ in \eqref{eq:Set_Omega}.
\end{proof}
Finally, we present the relationship between strong duality for $(P)-(D^F)$ and strong duality for $(P)-\left (\Db^F \right )$.
\begin{theorem}
\label{thm:SD_mixed}
Let $B$ and $\Lambda$ the sets defined in \eqref{eq:Set_B} and \eqref{eq:Set_Xi}, respectively, and assume that
\begin{equation} \label{eq:problemPe}
\epi(f-g+\delta_A)^c\cap B = \Lambda\cap B.
\end{equation}
 Then, the following statements are equivalent:
\begin{itemize}
	\item[i)] There exists strong duality for $(P)-(D^F)$;
	\item[ii)] There exists strong duality for $(P)-\left (\Db^F \right )$ and $\Omega\cap B=K\cap B$.
\end{itemize}
\end{theorem}
\begin{proof}
The implication $ii) \Rightarrow i)$ comes directly from Proposition \ref{prop:Strong_duality_via_Omega} $i)$ and \eqref{eq:problemPe} is not necessary.\\
$i) \Rightarrow ii)$ In first place, we will see which consequences equality \eqref{eq:problemPe} has.
Let us consider the primal problem
\begin{equation*}
	(P^e) ~~~ \inf_X \left\{ f-\eco g+\delta_A\right\}.
\end{equation*}
Let us observe that $\dom f \subseteq \dom g \subseteq \dom g^{cc'}$, and the function $f-\eco g$ is proper. Its Fenchel dual obtained following \cite{FVR2012} will be
\begin{equation*}	
	(D^e)~~~\sup_{\dom \delta_A^c}\left\{ -(f-\eco g)^c(-x^*,-y^*,\alpha)-\delta_A^c(x^*,y^*,\alpha)\right\}
\end{equation*}
and $v(P^e)\geq v(D^e)$. Due to the $\ep$-convexity of $g^c$, we can derive that $(\eco g)^c=(g^{cc'})^c=(g^c)^{c'c}=g^c$, and following along the lines of Section 3, it is possible to see that $(D^F)$ is equivalent to $(D^e)$, and $v(D^F)=v(D^e)$. Moreover, the dual problem for $(P^e)$ following the e-convexification technique in Section 3.1, named $(\Db^e)$, is equivalent to $ \left (\Db^F \right )$ and $v \left(\Db^F \right )= v \left (\Db^e \right )$.

Now, we will see that $v(P)=v(P^e)$. Clearly, since $\eco g\leq g$, $v(P^e)\geq v(P)$, but if $v(P^e)>v(P)$, we could find some $\beta\in\R$ such that
\begin{equation*}
	v(P)<-\beta\leq v(P^e).
\end{equation*}
Applying Lemma \ref{lem:Lemma1} $i)$, $(0,0,\delta,\beta)\in\epi(f-\eco g+\delta_A)^c$, for certain $\delta>0$, and by Proposition \ref{prop:Econv_formula}, the point $(0,0,\delta,\beta)\in\Lambda$. By equality \eqref{eq:problemPe}, $(0,0,\delta,\beta)\in\epi(f-g+\delta_A)^c$, so by Lemma \ref{lem:Lemma1} $i)$ $v(P)\geq-\beta$, which is a contradiction and, then, $v(P)=v(P^e)$. Now, because of the e-convexity of $\eco g$, we can derive that, from \eqref{eq:problemPe}, 
$$v(P) \geq v \left(\Db^F \right ) \geq v(D^F),$$
which means that weak duality holds for $(P)$ and $(\Db^F)$, equivalently, according to Proposition \ref{prop:NC_for_WD} and \eqref{eq:problemPe},
\begin{equation} \label{eq:inclusion_from_hip}
K\cap B \subseteq \Lambda \cap B.
\end{equation} 
Now, since strong duality between $(P)-(D^F)$ is equivalent to strong duality between $(P^e)-(D^e)$, using Theorem \ref{thm:SD_for_P0_D0}, we have
\begin{equation*}
\begin{aligned}	
	&\epi(f-\eco g+\delta_A)^c\cap B\\
	&= \bigcup_{\dom \delta_A^c} \Bigl \{ \epi\Bigl(f-\eco g-c(\cdot,(-x^*,-y^*,\alpha)) - \delta_A^c(x^*,y^*,\alpha)\Bigr)^c\cap B \Bigr \}
	\end{aligned}
\end{equation*}
and, due to Proposition \ref{prop:Econv_formula}, this means that
\begin{equation}\label{eq:omega_equal_lambda}
	\Lambda\cap B =\Omega\cap B. 
\end{equation}
Clearly, $\Omega \subseteq K$, then $\Omega \cap B \subseteq K \cap B$ and, from \eqref{eq:inclusion_from_hip} and \eqref{eq:omega_equal_lambda}, we have
\begin{equation*}
	\Lambda\cap B =\Omega\cap B=K\cap B, 
\end{equation*}
allowing us to conclude that $ii)$ is fulfilled, according to Proposition \ref{prop:Strong_duality_via_Omega} $ii)$ and \eqref{eq:problemPe}.
\end{proof}

Observe that equality \eqref{eq:problemPe} in Theorem \ref{thm:SD_mixed} does not imply that the function $g$ has to be e-convex in general.

\begin{example}
\label{ex:Condition_Th_SD_mixed}
Let $A=]0,+\infty[$ and the functions
\begin{equation*}
	f(x)=
	\left\{
	\begin{aligned}
		x+1, & ~~~x\geq 0,\\
		+\infty, &~~~ \text{otherwise;}
	\end{aligned}
	\right.
	\hspace{1cm}
	\text{and}
	\hspace{1cm}
	g(x)=
	\left\{
	\begin{aligned}
		1, & ~~~x= 0,\\
		x, & ~~~x>0,\\
		+\infty, & ~~~\text{otherwise.}
	\end{aligned}
	\right.
\end{equation*}
It is straightforward to see that $f$ is e-convex but $g$, though it is convex, is not e-convex according to Definition \ref{def:Econvex}. Given $x\in X$, if we calculate 
\begin{equation*}
	(f-g+\delta_A)(x)=(f-\eco g+\delta_A)(x)=
	\left\{
	\begin{aligned}
		1,&~~~x>0,\\
		+\infty,&~~~\text{otherwise,}
	\end{aligned}
	\right.
\end{equation*}
so, it is immediate to observe that 
\begin{equation*}
	\epi(f-\eco g+\delta_A)^c\cap B = \epi(f-g+\delta_A)^c\cap B,
\end{equation*}
although $g$ is not e-convex.
\end{example}

Next corollary comes directly from the above theorem and Proposition \ref{prop:Econv_formula}.

\begin{corollary}
Let us assume that the function $g$ in $(P)$ is e-convex. The following statements are equivalent:
\begin{itemize}
	\item[i)] There exists strong duality for $(P)-(D^F)$.
	\item[ii)] There exists strong duality for $(P)-\left(\Db^F \right )$ and $\Omega\cap B = K\cap B$.
\end{itemize}
\end{corollary}
\begin{remark}
As an additional comment, we point out {\normalfont{\cite{DVN2008}}} and {\normalfont{\cite{DNV2010}}}, where a constraint qualification called $(CC)$ is treated with the intention of finding regularity conditions for strong duality in DC problems. We leave as future research the analysis of the counterpart of this condition in our context when one of the functions is e-convex instead of convex and lsc. As it happens to be in the classical setting presented by Dihn et al. in the aforementioned works, this e-convex constraint qualification may be connected to those already developed in previous works like $(ECCQ)$ in {\normalfont{\cite{FVR2012}}} or $(C_{FL})$  in {\normalfont{\cite{FV2018}}}.
\end{remark}

\section{Stable Strong Duality}
Stable strong duality holds for a pair of primal and dual problem if strong duality holds for all the pairs of problems obtained by linearly perturbing the objective function of the primal problem and the corresponding dual problem.
If we perturbate the problem $(P)$, in the same way it was done in \cite{FV2016SSD}, taking $f(x)-g(x)$ as the objective function, we obtain, for each $p^* \in X^*$,

\begin{equation}
	\label{eq:Primal_pert_problem}
	\tag{${P}_{p^*}$}
		\begin{aligned}
		&\inf ~ f(x)-g(x)+ \langle x, p^* \rangle\\
		& \text{ s.t.} ~~~x\in A,
	\end{aligned}
\end{equation}
and its dual problem
\begin{equation*}
		\begin{aligned}
		\sup_{\substack{x^*,y^*\in X^*\\\alpha_1+\alpha_2>0}} \left\{ -(f-g)^c(-x^*-p^*,-y^*,\alpha_1) - \delta_A^c(x^*,y^*,\alpha_2)\right\}.
	\end{aligned}
\end{equation*}

Following the same steps as it was done in Section 3, we rewrite this dual  problem as
\begin{equation}
	\label{eq:Dual_pert_final}
	\tag{${D}_{p^*}^F$}
	\begin{aligned}
		\sup_{Z} ~\inf_{\dom g^c} \left\{ g^c(u^*,v^*,\gamma)-f^c(u^*-x^*-p^*,-y^*,\alpha) - \delta_A^c(x^*,y^*,\alpha)\right\}.
	\end{aligned}
\end{equation}

Applying the e-convexification technique also used in Subsection \ref{subsec:Econvexification}, we can obtain another dual problem for $(P_{p*})$ 
\begin{equation}\label{eq:dual_probl_pert_econ}
	\tag{${\Db}_{p^*}^F$}
	\inf_{\dom g^c} \sup_{Z} \left\{g^c(u^*,v^*,\gamma)-f^c(u^*-x^*-p^*,-y^*,\alpha)+-\delta_A^c(x^*,y^*,\alpha)\right\},
\end{equation}
verifying that, in case $g$ is e-convex, 
\begin{equation*}
	v(P_{p^*})\geq v \left({\Db}_{p^*}^F \right )  \geq v \left(D_{p^*}^F \right),
\end{equation*}
for all $p^* \in X^*$. In order to derive a characterization for stable strong duality, we need similar results to Lemma \ref{lem:Lemma1} and Proposition \ref{prop:NC_for_WD}, whose proofs will be analogous, taking, instead of the set $B$ in (\ref{eq:Set_B}), the following set, for each $p^* \in X^*$:
\begin{equation}
	\label{eq:Set_B_p}
	B_{p^*}:=\left\{ (-p^*,0)\times \R_{++}\times \R\right\} \subseteq X^* \times X^* \times \R\times \R.
\end{equation}

\begin{lemma}
\label{lem:Lemma2}
Let $p^*\in X^*$ and $\beta\in\R$. The following statements hold:
\begin{itemize}
	\item[i)] There exists $\delta>0$ such that $(-p^*,0,\delta,\beta)\in\epi(f-g+\delta_A)^c$ if and only if $v(P_{p^*})\geq-\beta$.
	\item[ii)] There exists $\delta>0$ such that $(-p^*,0,\delta,\beta)\in\Omega$ if and only if there exists $(x^*,y^*,\alpha)\in\dom\delta_A^c$ such that, for all $(u^*,v^*,\gamma)\in\dom g^c$,
	\begin{equation*}
		g^c(u^*,v^*,\gamma)-f^c(u^*-x^*-p^*,-y^*,\alpha)-\delta_A^c(x^*,y^*,\alpha)\geq-\beta.
	\end{equation*}
	\item[iii)]  There exists $\delta>0$ such that $(-p^*,0,\delta,\beta)\in K$, if and only if, for all $(x^*,y^*,\alpha)\in\dom\delta_A^c$, there exists $(u^*,v^*,\gamma)\in\dom g^c$ such that 
	\begin{equation*}
		g^c(u^*,v^*,\gamma)-f^c(u^*-x^*-p^*,-y^*,\alpha)-\delta_A^c(x^*,y^*,\alpha)\geq-\beta.
	\end{equation*}
\end{itemize}
\end{lemma}

\begin{proposition}
\label{prop:NC_for_WD_p}
Let $\Omega$ and $B_{p^*}$, for all $p^* \in X^*$, be the sets defined in \eqref{eq:Set_Omega} and \eqref{eq:Set_B_p}, respectively. Then: 
\begin{itemize}
\item [i)]$v(P_{p^*})\geq v(D_{p^*}^F)$ if and only if $\Omega\cap B_{p^*} \subseteq \epi(f-g+\delta_A)^c\cap B_{p^*}$.

\item[ii)] $v(P_{p^*})\geq v \left({\Db}_{p^*}^F \right)$ if and only if $K\cap B_{p^*}\subseteq \epi(f-g+\delta_A)^c\cap B_{p^*}$.
\end{itemize}

\end{proposition}

Both previous results lead us to the following chacterization of stable strong duality, whose proof is similar to that one from Proposition \ref{prop:Strong_duality_via_Omega}.

\begin{proposition}
\label{prop:Stable_Strong_duality_via_Omega}
Let $\Omega$, $K$ and $B_{p^*}$, for all $p^* \in X^*$,  be the sets defined by \eqref{eq:Set_Omega}, \eqref{eq:Set_K} and \eqref{eq:Set_B_p}  respectively. Then
\begin{itemize}
\item[i)] Strong duality for $(P_{p^*})$ and $(D_{p^*}^F)$ holds if and only if 
\begin{equation*}	
	\Omega\cap B_{p^*}  = \epi(f-g+\delta_A)^c\cap B_{p^*}.
\end{equation*}
\item[ii)] Strong duality for $(P_{p^*})-({\Db}_{p^*}^F)$ holds if and only if 
	\begin{equation*}
		K\cap B_{p^*}=\epi(f-g+\delta_A)^c\cap B_{p^*}.
	\end{equation*}
\end{itemize}

\end{proposition}

Denoting by
\begin{equation} \label{eq:Bstar}
 B^*:=X^* \times \{0\}\times \R_{++}\times \R \subseteq X^* \times X^* \times \R\times \R,
 \end{equation}
 next corollary is inmediate.

\begin{corollary}
Stable strong duality holds for $(P)-(D^F)$ \normalfont{(}and $(P)-(\Db^F)$, respectively\normalfont{)} if and only if  
\begin{align*}
	\Omega\cap B^* &= \epi(f-g+\delta_A)^c\cap B^*,\\
	(K\cap B^* &= \epi(f-g+\delta_A)^c\cap B^*, \textit{respectively}).
\end{align*} 
\end{corollary}

Theorem \ref{thm:SD_mixed} can also be extended to stable strong duality. For its proof, which follows similar steps, it is needed the next characterization of stable strong duality for $(P_0)-(D_0)$, which can be shown in a similar way than strong duality was in Theorem \ref{thm:SD_for_P0_D0}.

\begin{theorem}
Let $B^*$ be the set defined in \eqref{eq:Bstar}. There exists stable strong duality for $(P_0)-(D_0)$ if and only if
\begin{equation*}
	\epi(f+\delta_A)^c\cap B^* =\bigcup_{\dom \delta_A^c}\Bigl \{ \epi\Bigl(f-c(\cdot,(-x^*,-y^*,\alpha))-\delta_A^c(x^*,y^*,\alpha)\Bigr)^c\cap B^* \Bigr \}.
\end{equation*}
\end{theorem}

Now, we extend Theorem \ref{thm:SD_mixed}.
\begin{theorem}
Let $B^*$ be the set defined in \eqref{eq:Bstar} and $\Lambda$ defined in \eqref{eq:Set_Xi}. Assume that $\epi(f-g+\delta_A)^c\cap B^* = \Lambda\cap B^*$. Then, the following statements are equivalent:
\begin{itemize}
	\item[i)] There exists stable strong duality for $(P)-(D^F)$.
	\item[ii)] There exists stable strong duality for $(P)- \left(\Db^F \right)$ and $\Omega\cap B^*=K\cap B^*$.
\end{itemize}
\end{theorem}

\section{Conclusions}

In this paper we have derived a new Fenchel dual problem for a DC optimization primal one, $\inf_A(f-g)$, which does not necessarily verify weak duality in a general case. We have obtained necessary and sufficient conditions ensuring not only weak duality, but also strong duality. These conditions are expressed in terms of epigraphs and conjugate functions, so they belong to the class of closedness-type regularity conditions. We also develop an additional dual problem via a convexification technique whose optimal value, when the function $g$ is e-convex, ends up being bounded by the one of the primal and the first Fenchel dual problem that was obtained. Finally, we characterize not only the zero duality gap, but also strong and stable strong dualities for both dual pairs of optimization problems.

As future research we could study deeper properties of these regularity conditions writing down the problem from a general perspective via a perturbation function and following the approach by \cite{ET1976}. This could open new regularity conditions like the ones appearing in \cite{F2015} for Fenchel duality, \cite{FVR2016} for Lagrange duality or \cite{FV2018} for Fenchel-Lagrange dualities, which clearly goes beyond the scope of the current manuscript and may deserve to be studied.


\section*{Disclosure statement}

The authors declare that they have no conflict of interest.

\section*{Funding}

Research partially supported by MICIIN of Spain, Grant AICO/2021/165.

\bibliographystyle{plain}
\bibliography{biblio.bib}

\end{document}